\documentclass[a4paper,10pt]{article}
\usepackage[utf8]{inputenc}
\usepackage{amsthm}
\usepackage{amsfonts}
\usepackage{amsmath}
\usepackage{amssymb}

\newtheorem{theorem}{Theorem}[section]
\newtheorem{lemma}[theorem]{Lemma}
\newtheorem{corollary}[theorem]{Corollary}
\newtheorem{proposition}[theorem]{Proposition}
\newtheorem{defi}[theorem]{Definition}

\title{$*$-polynomial identities of $4 \times 4$ upper triangular matrices with the reflection involution}
\author{Ronald Ismael Quispe Urure 
\thanks{Supported by Ph.D. grant from CAPES}\\
Department of Mathematics, Federal University of S\~ao Carlos\\
13565-905 S\~ao Carlos, SP, Brazil\\
e-mail: \texttt{urure6@dm.ufscar.br}}

\begin{document}

\maketitle

\noindent\textbf{Keywords:}  Upper triangular matrix algebra, Involution, Identities with involution, PI-algebra.

\noindent\textbf{2010 AMS MSC Classification:} 16R10, 16R50, 16W10.

\

\begin{abstract}
	Let $UT_4(F)$ be $4\times 4$ upper triangular matrix algebra over a field $F$ 
	of characteristic zero and let $\mathcal{A}$ be the subalgebra of  $UT_4(F)$ linearly generated by $\{\mathbf{e}_{ij}:1 \leq i\leq j \leq 4 \} \setminus \mathbf{e}_{23}$ 
	where $\{\mathbf{e}_{ij} : 1 \leq i\leq j \leq 4\}$ is the standard basis of $UT_4(F)$. 
	We describe the set of all $*$-polynomial identities for $\mathcal{A}$ with the involution defined by the reflection of second diagonal.
\end{abstract}

\section{Introduction}
Let $F$ a field. 
 In this paper, we will consider unitary associative algebras over $F$ only. We refer to them simply by algebras. 

Let $\mathcal{R}$ be an algebra over a field $F$ of char(F)$\neq 2$ (characteristic different from $2$). 
A map
 $*:\mathcal{R} \rightarrow \mathcal{R}$ is an involution if it is an automorphism of the additive group $\mathcal{R}$ such that 
 \[(ab)^* = b^* a^* \ \ \mbox{and} \ \ (a^*)^* = a\]
  for all $a,b \in \mathcal{R}$. 
Let $Z(\mathcal{R})$ be the center of $\mathcal{R}$. If $a^*=a$ for all $a\in Z(\mathcal{R})$, then $*$ is called an involution of the first kind on $\mathcal{R}$. Otherwise $*$ is called an involution of the second kind.
From now on we consider involutions of the first kind only.
 
The description of the involutions on a given algebra is an important task in ring theory. 
In the algebra $UT_k(F)$ of the $k\times k$ upper triangular matrix over F we have an important involution.
For every matrix $A\in UT_k(F)$ define $A^* = J A^t J$ where $A^t$ denotes the usual matrix transpose and $J$ is the following permutation matrix:
\[ \begin{pmatrix}
0 & \cdots & 0 & 1 \\
0 & \cdots & 1 & 0 \\
\vdots & & \vdots & \vdots \\
1 & \cdots & 0 & 0
\end{pmatrix}. \] 
If $k$ is odd integer, any other involution in $UT_k(F)$ is completely determined by $*$ (see \cite{vinkossca}). $*$ is called of the \emph{reflection} or \emph{transpose involution} on $UT_k(F)$.
In the case $k=2l$ there exist two classes of inequivalent involutions. One of them is the same $*$ and the other is defined by $A^s = D A^* D$, for all $A \in UT_k(F)$. Here, $D$ is the matrix  
\[  \begin{pmatrix}
I_l & 0 \\
0 & -I_l 
\end{pmatrix} \] and $I_l$ is the identity matrix of the full matrix algebra $M_l(F)$. 
$s$ is called of the \emph{symplectic involution} on $UT_{2l}(F)$. 
Details about this results can be found in  \cite{vinkossca} and 
various others properties of involutions and involution-like maps for the upper
triangular matrix can be found in \cite{marcoux}. 

Let 
 $Y=\{y_1,y_2,\ldots \}$ and $ Z=\{z_1,z_2,\ldots \}$ be two disjoint infinite sets. 
 Denote by $F\langle Y \cup Z \rangle$ the free unitary associative algebra over $F$, freely generated by $Y\cup Z$. The elements of $F\langle Y \cup Z \rangle$ are called \emph{polynomials}.
Let $*$ be an involution for $\mathcal{R}$  and lets $\mathcal{R}^{+}$ and $\mathcal{R}^{-}$ denote the sets 
\[  \{a\in \mathcal{R}: \ a^*=a\} \qquad  \mbox{and} \qquad \{a\in \mathcal{R}: \ a^*=-a\} \]
respectively. 
  $f(y_1,\ldots,y_n,z_1,\ldots,z_m) \in F\langle Y \cup Z \rangle$ is a $*$-polynomial identity for $\mathcal{R}$ if 
\[f(a_1,\ldots,a_n,b_1,\ldots,b_m)=0\]
for all $a_1,\ldots,a_n \in \mathcal{R}^{+}$ and for all $b_1,\ldots,b_m\in \mathcal{R}^{-}$. 
Denote by $Id(\mathcal{R},*)$ the set of all  $*$-polynomial identities for  $(\mathcal{R},*)$.
It is natural to ask what the $*$-polynomial identities satisfied by the reflection  and by the symplectic involution of the algebra $UT_k(F)$. 
There is a description of $Id(UT_2(F),*)$ and $Id(UT_2(F),s)$ when $F$ is an arbitrary field of char$(F) \neq 2$, see \cite{vinkossca} for case $F$ infinite and \cite{urure} for case $F$ finite. Have also been described $Id(UT_3(F),*)$ when $F$ is a field of char$(F)=0$. It's an open problem describe $Id(UT_k(F),*)$ in other cases.

Consider the subalgebra $\mathcal{A}$  of $UT_4(F)$ constituting by matrices 
\[ 
\begin{pmatrix}
A & C \\
0 & B
\end{pmatrix} \] where $A,B,C \in UT_2(F)$.
In order to solve the $k=4$ case, in this paper we describe the set of all $*$-polynomial identities for $\mathcal{A}$ with this involution.

\section{Preliminaries}
$UT_4(F)^+$ and $UT_4(F)^-$ are subspaces of $UT_4(F)$. If char$F \neq 2$, can be write as direct sum   $$UT_4(F) = UT_4(F)^+ \oplus UT_4(F)^-.$$
The center of the $UT_4(F)$ is given by the scalar matrices 
$ \{ \lambda I_4 : \ \lambda \in F \}$.

   Given $\mathcal{S}$  an algebra, the commutators in $S$ 
are defined inductively by
\begin{equation*}
	[a_1,a_2] = a_1 a_2-a_2 a_1 \qquad \mbox{and} \qquad [a_1,\ldots,a_n] = [[a_1,\ldots,a_{n-1}],a_n]. 
\end{equation*} for all $a_1,\ldots,a_n \in \mathcal{S}$. $\mathcal{S}$ satisfy the Jacobi's Identity, it is
\[  [a_3,a_2,a_1] +  [a_2,a_1,a_3] + [a_1,a_3,a_2] = 0 \] 
for all $a_1,a_2,a_3 \in \mathcal{S}$.

Let 
$X=\{x_1,x_2,\ldots \}$  be an infinite set. 
Denote by $F\langle X \rangle$ the free unitary associative algebra over $F$, freely generated by $X$. 
$f(x_1,\ldots,x_n) \in F\langle X \rangle$ is a polynomial identity for $\mathcal{S}$ if 
\[f(a_1,\ldots,a_n)=0\]
for all $a_1,\ldots,a_n \in \mathcal{S}$.
Denote by $Id(\mathcal{S})$ the set of all  polynomial identities for  $\mathcal{S}$.
A $T$-ideal of $F\langle X \rangle$ is an ideal closed under all endomorphism of $F\langle X \rangle$. Recall that this is equivalent to stating that an ideal $I$ is a $T$-ideal of $F\langle X \rangle$ if
$f(x_1,\ldots,x_n) \in I$ then
\[f(g_1,\ldots,g_n) \in I\]
for all $g_1,\ldots,g_n \in F\langle X \rangle$.

 $F\langle Y \cup Z \rangle$ has an involution, which we denote by $*$ as well, satisfying  
 \[ (y_i)^* = y_i \ \ \mbox{and} \ \ (z_i)^* = -z_i. \]
 for all $i \geq 1$. 
 The elements of $Y$ are called of symmetric variables and the elements of $Z$ are called of skew-symmetric variables.
An endomorphism $\varphi$ of $F\langle Y \cup Z \rangle$ preserves involution if 
\[ \varphi (f^*) = (\varphi(f))^* \]
for all $f\in F\langle Y \cup Z \rangle$. 
An ideal $I$ is a $*$-ideal if $a\in I$ implies $a^* \in I$.
A $T(*)$-ideal of $F\langle Y \cup Z \rangle$ is an ideal closed under all endomorphism of $F\langle Y \cup Z \rangle$ that preserves involution. Recall that this is equivalent to stating that a $*$-ideal $I$ is a $T(*)$-ideal if
%We have to that 
 $f(y_1,\ldots,y_n,z_1,\ldots,z_m) \in I$ then
\[f(g_1,\ldots,g_n,h_1,\ldots,h_m) \in I\]
for all $g_1,\ldots,g_n \in F\langle Y \cup Z \rangle^{+}$ and all $h_1,\ldots,h_m \in F\langle Y \cup Z \rangle^{-}$. A subset $S\subseteq I$ generates $I$ as a $T(*)$-ideal if $I$ is a minimal $T(*)$-ideal containing $S$. 

Let  $\mathcal{R}$ an algebra with involution $*$
	$Id(\mathcal{R},*)$ is a $T(*)$-ideal of $F\langle Y \cup Z \rangle$. Conversely, each $T(*)$-ideal of $F\langle Y\cup Z \rangle$ is the set of $*$-polynomial identities for some algebra with involution $(\mathcal{R},*)$.

A polynomial $f(y_1,\ldots,y_n,z_1,\ldots,z_m) \in F\langle Y\cup Z \rangle$ is 
called $Y$-proper if $f$ is a linear combination of polynomials
\[z_1^{r_1}\cdots z_m^{r_m}f_1 \cdots f_t \]
where $r_i \geq 0$, $f_i \in F\langle Y \cup Z \rangle$ is a  
commutator, $t\geq 0$ ($f_0=1$ if $t = 0$). 
Denote by $B$ the vector space of all $Y$-proper polynomials.
Every element 
$g(y_1,\ldots,y_n,z_1,\ldots,z_m) \in F\langle Y \cup Z \rangle$ is a linear combination of polynomials 
\begin{equation}
y_1^{s_1}\cdots y_n^{s_n}g_{(s_1,\ldots,s_n)} 
\end{equation}
where $s_i \geq 0$ and $g_{(s_1,\ldots,s_n)}\in B$. When $F$ is infinite it follows that every $T(*)$-ideal is generated by its $Y$-proper polynomials and when $F$ is a field of char$F=0$ every $T(*)$-ideal is generated by its $Y$-proper multilinear ones. See \cite[Lemma 2.1]{drenskygiambruno}

 Lemma \ref{lema base de Udos} can be found in \cite[Theorem 5.2.1]{drenskybook}. 

\begin{lemma}\label{lema base de Udos}
	Let $F$ be an infinity field. Then, a linear basis for $F \left\langle X \right\rangle / Id(UT_2(F))$ is given by the elements 
	\begin{align*} 
	x_1^{r_1} \ldots x_m^{r_m} [x_{j_1},\ldots,x_{j_n}]^t + Id(UT_2(F))
	\end{align*} where, $r_i \geq 0$, $t\in \{0,1\}$, $j_1 > j_2 \leq \ldots \leq j_n.$
\end{lemma}

Consider the following order on variables: 
\[ z_1 < z_2 < \ldots < z_m < \ldots < y_1 < y_2 < \ldots < y_n < \ldots.\]
	The Lemma \ref{lema comutador ordenado} is an adaptation of \cite[Theorem 5.2.1]{drenskybook}.
\begin{lemma}\label{lema comutador ordenado}
	Let $w_i \in Y \cup Z$ and $u = [w_{i_1},\ldots,w_{i_n}]$ be a commutator in $F\left\langle Y\cup Z\right\rangle $, then $u = v + v'$ where $v$ is a linear combination of polynomials of type
	\[ [w_{j_1},\ldots,w_{j_n}] \] with $w_{j_1} > w_{j_2} \leq \ldots \leq w_{j_n}$ and $v'$ is a linear combination of products of two commutators in $F\left\langle Y\cup Z\right\rangle $. 
\end{lemma}

 Denote by $\mathcal{B}$ the subalgebra of $UT_4(F)$ defined by the matrices 
 \begin{equation*}
 \left(
 \begin{array}{cc}
 A&0 \\
 0&B
 \end{array}
 \right)
 \end{equation*} where $A,B \in UT_2(F)$. 	Obviously, 
 \begin{align} \label{lema B}	
 Id(UT_2(F)) \subseteq Id(\mathcal{B})
 \end{align}
 
 Proposition \ref{propo base de B} describe the $*$-polynomial identities of $\mathcal{B}$.
\begin{proposition}\label{propo base de B}
	Let $F$ be an infinity field. Then, a  basis for $B / (B \cap Id(\mathcal{B},*))$ is given by
\begin{align}\label{base de B}
	z_1^{r_1} \ldots z_m^{r_m} [w_{j_1},\ldots,w_{j_s}]^{\theta} + B \cap Id(\mathcal{B},*)
	\end{align} where $r_i \geq 0$, $\theta \in \{0,1\}$, $w_{j_1} > w_{j_2} \leq \ldots \leq w_{j_s}.$
	\end{proposition}

	\begin{proof}
	By   (\ref{lema B}) we have to the product of two commutators in $F \left\langle Y \cup Z \right\rangle$ is contained in $Id(\mathcal{B},*)$. Next, by Lemma \ref{lema comutador ordenado}, we have to each $Y$-proper polynomial, module $Id(\mathcal{B},*)$, is a  linear combination of elements in (\ref{base de B}).

	Let $f(z_1,\ldots,z_m,y_1,\ldots,y_n) \in F\left\langle Y \cup Z \right\rangle $ be a linear combination of elements in (\ref{base de B}) such that $f\in Id(\mathcal{B},*)$. Put $$x_1 = z_1, \ldots, x_m = z_m, x_{m+1} = y_1, \ldots, x_{m+n} = y_n,$$ $f$ can be written in the form:
	\[ f = \sum_{r,j} \alpha_{r,j} z_1^{r_1} \ldots z_m^{r_m} [x_{j_1},\ldots,x_{j_s}] + \sum_{r} \alpha_r z_1^{r_1} \ldots z_m^{r_m}  \] 
	where $\alpha_{r,j}, \alpha_r \in F$, $r = (r_1,\ldots,r_m)$, $j = (j_1,\ldots,j_s)$ and ${j_1} > {j_2} \leq \ldots \leq {j_s}$.
	Let $A_i,B_k \in UT_2(F)$, then 
	\[ Y_i = 
	\begin{pmatrix}
	A_i & 0 \\
	0 & A_i^*
	\end{pmatrix} \in \mathcal{B}^+ \ \ \mbox{and} \ \ Z_k = 
	\begin{pmatrix}
	B_k & 0 \\
	0 & - B_k^*
	\end{pmatrix} \in \mathcal{B}^-. \]
	By substituting in $f$, we have to $f(Z_1,\ldots,Z_m,Y_1,\ldots,Y_n) = 0$.
Thus,	it follows that  $f(B_1,\ldots,B_m,A_1,\ldots,A_n) = 0$. Since $A_i$ and $B_k$ are arbitrary, we have to $f(x_1, \ldots, x_m,x_{m+1},\ldots,x_{m+n} ) \in Id(UT_2(F))$ seen as element of $F\left\langle X \right\rangle $. By Lemma \ref{lema base de Udos} the proof is complete.
\end{proof}

\section{$*$-identities for $\mathcal{A}$}

Let $A_1,A_2,A_3 \in UT_2(F)$, it is easy to check that:
\begin{itemize}
	\item[i)] $[A_1,A_2]^* = [A_1,A_2]$.
	\item[ii)] $[A_1,A_2] (A_3 - A_3^*) = [A_1,A_2,A_3]$.
	\item[iii)] If $C \in UT_2(F)^+$ then $A_1 C - C A_1^* = \lambda (A_1 - A_1^*)$, where $\lambda = 2^{-1}tr (C)$ (Here $tr(C)$ is the sum of the diagonal elements of $C$).
\end{itemize}

If  $Y \in \mathcal{A}^+$  then
\[ Y = 
\begin{pmatrix}
A & C \\
0 & A^*
\end{pmatrix}\] for some $A,C\in UT_2(F)$ with $C^* = C$. If 
$Z \in \mathcal{A}^-$ then
\[Z =
\begin{pmatrix}
B & D \\
0 & - B^*
\end{pmatrix}\]
for some  $B,D\in UT_2(F)$ with $D^* = -D$.

Denote by $\mathbf{e}_{ij}$, the element of $M_2(F)$ with exactly one nonzero entry ($i$-row and $j$-column), which is $1$. 
%The Lemma \ref{lema das identidades de A} and the Lemma \ref{lema identidade geral de Y} are easy to verify.
We are going to give some facts about of the $\mathcal{A}$.
\begin{lemma}\label{lema das identidades de A}
	Let $P_i$, $1 \leq i \leq 6$,
	be arbitrary elements of $\mathcal{A}$. 
	Then 
	\begin{itemize}
		\item[i)]  The matrices 
		$$[P_1,P_2] [P_3,P_4],\quad [P_1,P_2] P_3 [P_4,P_5]\quad \mbox{and} \quad [P_1,P_2] [P_3,P_4] P_5$$ are of the type $
		\begin{pmatrix}
		0 & \alpha \mathbf{e}_{12} \\
		0 & 0
		\end{pmatrix}$, for some $\alpha \in F$. 
		\item [ii)] $[P_1,P_2][P_3,P_4][P_5,P_6] = 0$.
	\end{itemize}
\end{lemma}

\begin{proof} $i)$
	The commutators  $[P_1,P_2]$ and $[P_1,P_2] P_3$ are matrices of the type 
	\begin{equation}\label{produto de dois comutadores A}
	\begin{pmatrix} \alpha \mathbf{e}_{12} & \Theta \\ 0 & \beta \mathbf{e}_{12} \end{pmatrix}
	\end{equation}
	for some $\alpha,\beta \in F$, $\Theta \in UT_2(F)$. 
	Product of two matrices of type (\ref{produto de dois comutadores A}) is given by 
	\begin{equation*} 
	\begin{pmatrix}
	\alpha_1 \mathbf{e}_{12} & \Theta_1 \\
	0 & \beta_1 \mathbf{e}_{12}
	\end{pmatrix}
	\begin{pmatrix}
	\alpha_2 \mathbf{e}_{12} & \Theta_2 \\
	0 & \beta_2 \mathbf{e}_{12}
	\end{pmatrix} = 
	\begin{pmatrix}
	0 & \Theta \\
	0 & 0
	\end{pmatrix}\end{equation*} where $\Theta = \alpha_1 \mathbf{e}_{12} \Theta_2 + \Theta_1 \beta_2 \mathbf{e}_{12} = \alpha \mathbf{e}_{12}$, for some $\alpha \in F$.  
	
	\noindent $ii)$ Product of three matrices of type (\ref{produto de dois comutadores A}) is 
	\[\begin{pmatrix}
	0 & \alpha \mathbf{e}_{12} \\
	0 & 0
	\end{pmatrix}
	\begin{pmatrix}
	\alpha_3 \mathbf{e}_{12} & \Theta_3 \\
	0 & \beta_3 \mathbf{e}_{12}
	\end{pmatrix} = 0. \] 
	The proof is complete.
\end{proof}

\begin{lemma}\label{lema identidade geral de Y}
	Let $Y  \in \mathcal{A}^+$. Then, there is  $\alpha \in F$ such that
	\[ Q_1 Y Q_2 = \alpha Q_1 Q_2. \] for all $Q_1,Q_2 \in \mathcal{A}$ where 
	$$Q_i = \begin{pmatrix} \alpha_i \mathbf{e}_{12} & \Theta_i \\ 0 & \beta_i \mathbf{e}_{12} \end{pmatrix}$$
	and $\alpha_i,\beta_i \in F$, $\Theta_i \in UT_2(F)$, $i=1,2$.
	
\end{lemma}
\begin{proof} Let $Y = \begin{pmatrix}	A & C \\ 0 & A^* \end{pmatrix}$ where
	$A = a \mathbf{e}_{11} + b \mathbf{e}_{22} + c \mathbf{e}_{12} $
	% =\begin{pmatrix} a & c \\ 0 & b \end{pmatrix} 	\in UT_2(F)$ 
	and $C \in UT_2(F)$. We can write  $Y = Y' + Y''$ where
	\[ Y' = 
	\begin{pmatrix}
	a \mathbf{e}_{11} + c \mathbf{e}_{12} & C \\
	0 & a \mathbf{e}_{22} + c \mathbf{e}_{12}
	\end{pmatrix} 
	\ \ \mbox{and} \ \ Y'' = 
	\begin{pmatrix} b \mathbf{e}_{22} & 0 \\ 0 & b \mathbf{e}_{11} \end{pmatrix}.
	\] We have to
	$Q_1 Y' = 
	\begin{pmatrix}
	0 & \Theta \\
	0 & \beta_1 a \mathbf{e}_{12}
	\end{pmatrix}$ with $\Theta$ satisfying $\Theta = \Theta \mathbf{e}_{22}.$	
	Thus,  $Q_1 Y' Q_2 = 0$.
	It is easy to verify that $Q_1 Y'' Q_2 = b Q_1 Q_2.$ 
	This proof is complete.
\end{proof}

\begin{proposition}\label{propo identidades A}
	Let $v_1,v_2,v_3$  be commutators of $F \left\langle Y \cup Z \right\rangle$ and let $w \in Y \cup Z$. Then, the  following polynomials \[ v_1 v_2 v_3 \qquad  v_1 w v_2 - (v_1 w v_2)^{*} \qquad  v_1 v_2 w - (v_1 v_2 w)^{*}\] are $*$-polynomial identities for $\mathcal{A}$. In particular, $v_1 v_2 - v_2^{*} v_1^{*} \in Id(\mathcal{A},*)$.
\end{proposition}
\begin{proof}
The first polynomial is $*$-identity for $\mathcal{A}$ by item $ii)$ of the Lemma \ref{lema das identidades de A}. 
Let $P_i \in \mathcal{A}^+ \cup \mathcal{A}^-$, then by item $i)$ of the Lemma \ref{lema das identidades de A} we obtain \[ ([P_1,P_2] P_3 [P_4,P_5])^* = [P_1,P_2] P_3 [P_4,P_5], \] \[   ([P_1,P_2] [P_3,P_4] P_5)^* = [P_1,P_2][P_3,P_4] P_5.\] 
\end{proof}

\begin{proposition}\label{identidade Y}
	The polynomial \[ [y_4,y_3][y_2,y_1] + [y_3,y_2][y_4,y_1] + [y_2,y_4][y_3,y_1]\] is $*$-polynomial identity for $\mathcal{A}$.
	\end{proposition}
	\begin{proof}
	Let $Y_i = 
	\begin{pmatrix}
	A_i & C_i \\
	0 & A_i^*
	\end{pmatrix}\in \mathcal{A}^+$, $1\leq i \leq 4$. We have to
	\[ [Y_2,Y_1] = 
	\begin{pmatrix}
	[A_2,A_1] & \Lambda_{21} \\
	0 & - [A_2,A_1]
	\end{pmatrix}\] where $\Lambda_{21} = A_2 C_1 + C_2 A_1^* -A_1 C_2 - C_1 A_2^* = \lambda_1 (A_2 - A_2^*) - \lambda_2 (A_1 - A_1^*)  $ where $\lambda_i = 2^{-1} tr (C_i)$. Thus,
	\[ [Y_4,Y_3][Y_2,Y_1] 
	=
	\begin{pmatrix}
	0 & \Theta_{2} \\
	0 & 0
	\end{pmatrix}\] 
	where $\Theta_{2} = [A_4,A_3] \Lambda_{21} - \Lambda_{43} [A_2,A_1]$.
	By after some manipulations we obtain
	\begin{align*}
	\Theta_{2} = \lambda_1 [A_4,A_3,A_2] + \lambda_2 [A_3,A_4,A_1] + \lambda_3 [A_2,A_1,A_4] + \lambda_4 [A_1,A_2,A_3].
	\end{align*} Put
	\[ [Y_2,Y_4][Y_3,Y_1] = \begin{pmatrix}
	0 & \Theta_3 \\
	0 & 0
	\end{pmatrix} \ \ \mbox{and} \ \ [Y_3,Y_2][Y_4,Y_1] = \begin{pmatrix}
	0 & \Theta_4 \\
	0 & 0
	\end{pmatrix}.
	\]  By Jacobi's identity we obtain that 
	 $\Theta_2 + \Theta_3 + \Theta_4 = 0$.
\end{proof}

Let $f(y_1,y_2,y_3,w_1,\ldots, w_m) \in F\langle Y\cup Z \rangle$ where $w_1 \ldots w_m \in Y$ and let $$\mathbf{Ja}_{(y_1,y_2,y_3)} f$$ denote the polynomial
\[ 	 f(y_1,y_2,y_3,w_1,\ldots, w_m) + f(y_2,y_3,y_1,w_1,\ldots, w_m)  + f(y_3,y_1,y_2,w_1,\ldots, w_m).
\]

\begin{corollary}\label{coro with w}
Let $w_1,w_2 \in Y$. The  polynomial 
$$ \mathbf{Ja}_{(y_1,y_2,y_3)} [y_1,y_2,w_1] [y_3,w_2] $$
is $*$-polynomial identity for $\mathcal{A}$.  
\end{corollary}

\begin{proof} For instance $[y_i,y_j,w_1] = [y_i,y_j]w_1 - w_1 [y_i,y_j]$. Thus, we have to
\begin{align*}
&\mathbf{Ja}_{(y_1,y_2,y_3)} ([y_1,y_2,w_1] [y_3,w_2]) =\\ 
&= \mathbf{Ja}_{(y_1,y_2,y_3)} ([y_1,y_2] w_1 [y_3,w_2]) - w_1\mathbf{Ja}_{(y_1,y_2,y_3)} ([y_1,y_2] [y_3,w_2]).
\end{align*} Let $Y_1,Y_2,Y_3,W_1,W_2$, be elements of $\mathcal{A}^+$. By Lemma \ref{lema identidade geral de Y} there exists $\alpha \in F$ such that
\[ [Y_i,Y_j] W_1 [Y_k,W_2] = \alpha [Y_i,Y_j] [Y_k,W_2] \]
for all $i,j,k \in \{ 1,2,3 \}$. Therefore
\[ \mathbf{Ja}_{(y_1,y_2,y_3)} ([Y_1,Y_2] W_1 [Y_3,W_2]) = \alpha \mathbf{Ja}_{(y_1,y_2,y_3)} ([Y_1,Y_2] [Y_3,W_2]).\] By Proposition \ref{identidade Y} we obtain that $\mathbf{Ja}_{(y_1,y_2,y_3)} ([Y_1,Y_2] [Y_3,W_2]) = 0$. Thus, it has been demonstrated. 
\end{proof}

\begin{defi}
	Let  $I$ denote  the $T(*)$-ideal of $F\left\langle Y \cup Z \right\rangle $
	generated by the following polynomials:
	\begin{itemize}
		\item[1)] $v_1 v_2 v_3$, such that $v_1, v_2, v_3$ be commutators.
		\item[2)] $v_1 u v_2 - v_2^* u^* v_1^*$, $v_1 v_2 u - u^* v_2^* v_1^*$, such that $v_1, v_2$ be commutators and $u \in Y \cup Z$.
		\item[3)] $ \mathbf{Ja}_{(y_1,y_2,y_3)} ([y_1,y_2] [y_3,y_4])$, $ \mathbf{Ja}_{(y_1,y_2,y_3)} ([y_1,y_2,y_4] [y_3,y_5])$.
	\end{itemize}
\end{defi}

	By Proposition \ref{propo identidades A}, Proposition \ref{identidade Y} and Corollary \ref{coro with w} we have to
	\begin{equation} 
\label{primeira inclusao}
	 I \subseteq Id(\mathcal{A},*).	\end{equation}

\begin{lemma}
\label{lema identidade 2z}
For every $v_1$,$v_2$ commutators of $F\left\langle Y \cup Z \right\rangle $ 
	there exists $\alpha  \in F$ such that \[ 2 z_1 v_1 v_2 + [v_1,z_1] v_2 + \alpha [v_2,z_1] v_1 \in I.\]
	\end{lemma}
	\begin{proof} 
	It is easy to verify that given a commutator $v$, either $v^* = v$ or $v^* = - v$. Thus, there exists	 $\alpha \in F$ such that
	\begin{itemize}\label{comutadores comutam}
		\item [1)] $v_2^* v_1^* = \alpha v_2 v_1$
		\item [2)] $v_1 v_2 - \alpha v_2 v_1 + I  =   I$
		\item [3)] $v_1 z_1 v_2 + \alpha v_2 z_1 v_1 + I =  I$.
	\end{itemize}
Therefore, $z_1 v_1 v_2 + v_1 z_1 v_2 + \alpha [v_2,z_1] v_1 + I = I$. The proof is complete.
	\end{proof}

\begin{lemma}\label{lema z comuta}
	Let 
 $f$ a polynomial of $F \left\langle Y \cup Z \right\rangle$ and lets $v_1,v_2,v_3$ commutators of $F \left\langle Y \cup Z \right\rangle$, then $v_1 f v_2 v_3 \in I$.
\end{lemma}
\begin{proof}
Let  $w \in Y\cup Z$. Then, \[ v_1 w v_2 v_3 + I = w v_1 v_2 v_3 + [v_1,w] v_2 v_3 + I = I.\] Since $f = f^+ + f^-$ the proof is complete.
\end{proof}

\begin{proposition}\label{propo identidade geral de Y}
	Let $w_1,\ldots,w_{m+1}$ be symmetric variables. Then
	\[ \mathbf{Ja}_{(y_1,y_2,y_3)} ([y_1,y_2,w_1,\ldots,w_{m}] [y_3,w_{m+1}]) \in I \] 
 for all $m\geq 1$.
\end{proposition}
\begin{proof}
This proof is done by induction.	The equality $a[b,c] = [b,a c] - [b,a] c$ is satisfy by elements of 
$F \left\langle Y \cup Z \right\rangle$. Then, we have to 
\begin{align*}
	& \mathbf{Ja}_{(y_1,y_2,y_3)} ([y_1,y_2] [y_3,w_1w_2]) + I =\\ 
	& =\mathbf{Ja}_{(y_1,y_2,y_3)} ([y_1,y_2] w_1[y_3,w_2])  + I \\ 
	& = w_1 \mathbf{Ja}_{(y_1,y_2,y_3)} ([y_1,y_2] [y_3,w_2]) + \mathbf{Ja}_{(y_1,y_2,y_3)} ([y_1,y_2,w_1] [y_3,w_2]) + I = I.
\end{align*}
Suppose that  $m\geq 2$ and that
\[ \mathbf{Ja}_{(y_1,y_2,y_3)} ([y_1,y_2,w_1,\ldots,w_{k}] [y_3,w_{m+1}]) + I = I\]
\begin{equation}\label{induction} 
\mathbf{Ja}_{(y_1,y_2,y_3)} ([y_1,y_2,w_1,\ldots,w_{k-1}] [y_3,w_k w_{m+1}]) + I = I\end{equation}
 for all $1 \leq k <m$.
It follows that
 \begin{align}\label{j1}
 	\mathbf{Ja}_{(y_1,y_2,y_3)} (w_{k+1} [y_1,y_2,w_1,\ldots,w_{k}]  [y_3,w_{m+1}]) = \nonumber \\
 	= w_{k+1} \mathbf{Ja}_{(y_1,y_2,y_3)} ([y_1,y_2,w_1,\ldots,w_{k}] [y_3,w_{m+1}]) \in I
 \end{align}
 and
 \begin{align}\label{j2}
\mathbf{Ja}_{(y_1,y_2,y_3)} ( [y_1,y_2,w_1,\ldots,w_{k}]  [y_3,w_{k+1}] w_{m+1}) = \nonumber \\
= (\mathbf{Ja}_{(y_1,y_2,y_3)} [y_1,y_2,w_1,\ldots,w_{k}] [y_3,w_{k+1}]) w_{m+1} \in I.
\end{align}

By apply  (\ref{j1}) and (\ref{j2}) we have to
\begin{align}\label{equations}
	&\mathbf{Ja}_{(y_1,y_2,y_3)} ([y_1,y_2,w_1,\ldots,w_{m}] [y_3,w_{m+1}]) + I = \nonumber \\ &= \mathbf{Ja}_{(y_1,y_2,y_3)} ([y_1,y_2,w_1,\ldots,w_{m-1}]  w_{m} [y_3,w_{m+1}]) + I \nonumber \\
	&= \mathbf{Ja}_{(y_1,y_2,y_3)} ([y_1,y_2,w_1,\ldots,w_{m-1}]  [y_3, w_{m} w_{m+1}]) + I  \nonumber \\
	&= \mathbf{Ja}_{(y_1,y_2,y_3)} ([y_1,y_2,w_1,\ldots,w_{m-2}] w_{m-1} [y_3, w_{m} w_{m+1}]) \nonumber\\
		&+ w_{m-1} \mathbf{Ja}_{(y_1,y_2,y_3)} ([y_1,y_2,w_1,\ldots,w_{m-2}] [y_3, w_{m} w_{m+1}]) + I \nonumber\\
		&=  \mathbf{Ja}_{(y_1,y_2,y_3)} ([y_1,y_2,w_1,\ldots,w_{m-2}] [y_3, w_{m-1} w_{m} w_{m+1}]) + I
\end{align}
The element $w_{m-1} w_m$ can be written as $$2^{-1} (w_{m-1} w_m + w_{m-1} w_m) + 2^{-1} [w_{m-1}, w_m] .$$
Note that   $w_{m-1} w_m + w_{m-1} w_m \in F \left\langle Y \cup Z \right\rangle^+$ and that any product of three commutators is in $I$. Then, (\ref{equations}) is equal to
\begin{align*}
	=  2^{-1} 	\mathbf{Ja}_{(y_1,y_2,y_3)} ([y_1,y_2,w_1,\ldots,w_{m-2}] [y_3, [w_{m-1}, w_{m}]]) w_{m+1} + I
\end{align*}
So by induction  hypothesis, equation (\ref{induction}), this element is equal to $I$, also
\[ \mathbf{Ja}_{(y_1,y_2,y_3)} ([y_1,y_2,w_1,\ldots,w_{m-1}]  [y_3, w_{m} w_{m+1}]) \in I. \] 
\end{proof}

\begin{proposition}\label{propo dois comutadores}
Let  $\delta,\gamma,\epsilon$ be integers such that $\delta \geq 0$, $\gamma,\epsilon \geq 2$. Let $u_j \in Y \cup Z$ where $i \in \mathbb{N}$ and $$f = z_{r_1} \ldots z_{r_{\delta}} [u_{s_1},\ldots,u_{s_{\gamma}}] [u_{t_1},\ldots,u_{t_{\epsilon}}].$$  Then, $f + I$ is linear combination of elements the type
	\[ u_{i_1} \ldots u_{i_m} [u_{j_1},\ldots,u_{j_n}][u_{k_1},u_{k_2}] + I,\] where
	\begin{itemize}		
		\item[1)] $m \geq 0$, $u_{k_2} < u_{k_1}$, $u_{j_1} > u_{j_2} \leq \ldots \leq u_{j_n}$ and  $u_{i_1} \leq \ldots \leq u_{i_m}$.
		\item[2)] $u_{k_2} \leq u_{i_1}, u_{j_2}$.
		\item[3)] Every $u_{i_1} \ldots u_{i_m} [u_{j_1},\ldots,u_{j_n}][u_{k_1},u_{k_2}]$ has the same multidegree as $f$.
	\end{itemize}
\end{proposition}
\begin{proof}
Let $u_{k_2} = \min \{z_{r_1}, \ldots, z_{r_{\delta}}, u_{s_1},\ldots,u_{s_{\gamma}}, u_{t_1},\ldots,u_{t_{\epsilon}}\}$. We have two cases:
	
	\
	
	\noindent \textbf{Case} $\delta = 0$. In this case $f = [u_{s_1},\ldots,u_{s_{\gamma}}] [u_{t_1},\ldots,u_{t_{\epsilon}}]$. By (\ref{comutadores comutam}) we can suppose that $u_{k_2}$ it is in the second commutator of $f$. In addition, by Lemma \ref{lema comutador ordenado}  we can suppose that $k_2 = t_2  $ and $u_{t_1} > u_{t_2} \leq u_{t_3} \leq \ldots \leq u_{t_{\epsilon}}$.
Write $$[u_{t_1},\ldots,u_{t_{\epsilon}}] = \sum_{m_1,m_2}  m_1 [u_{k_1},u_{k_2}] m_2,$$ where, $m_1,m_2$ monomials such that  $m_1 [u_{k_1},u_{k_2}] m_2$ has the same multidegree as $[u_{t_1},\ldots,u_{t_{\epsilon}}]$.
	
	Now, for every $m_1$, we have to that $ [u_{s_1},\ldots,u_{s_{\gamma}}] m_1$  is linear combination of elements the type $u_{i_1} \ldots u_{i_l} [u_{j_1},\ldots,u_{j_n}]$  of the same multidegree as $ [u_{s_1},\ldots,u_{s_{\gamma}}] m_1$. Then,  $f + I$ is linear combination of  elements of the type
	\[ u_{i_1} \ldots u_{i_l} [u_{j_1},\ldots,u_{j_n}][u_{k_1},u_{k_2}] m_2 + I. \]
	By (\ref{comutadores comutam}) again,  $f + I$ is linear combination of elements of the type
\[ u_{i_1} \ldots u_{i_m} [u_{j_1},\ldots,u_{j_n}][u_{k_1},u_{k_2}] + I, \] where each $u_{i_1} \ldots u_{i_m} [u_{j_1},\ldots,u_{j_n}][u_{k_1},u_{k_2}]$ has the same multidegree as $f$. Finally, by Lemma \ref{lema comutador ordenado} we can suppose that  $u_{j_1} > u_{j_2} \leq \ldots \leq u_{j_n}$ and  $u_{i_1} \leq \ldots \leq u_{i_m}$.
	
	\
	
	\noindent \textbf{Case} $\delta >0$. In this case $u_{k_2} = z_r$ for some $r$. By Lemma \ref{lema z comuta} and  Lemma \ref{lema identidade 2z} we can suppose, without loss of generality, that $u_{k_2}$ it is in $[u_{s_1},\ldots,u_{s_{\gamma}}] [u_{t_1},\ldots,u_{t_{\epsilon}}]$. It is enough to apply the preceding case to $ [u_{s_1},\ldots,u_{s_{\gamma}}] [u_{t_1},\ldots,u_{t_{\epsilon}}]$ and, if necessary, reordering the variables that are outside of the commutators. 
	
	So, we conclude the proof.
	\end{proof}

\begin{corollary}\label{coro so variaveis Y}
	Let $v_1, v_2 $ be commutators involving the symmetric variables only, then $v_1 v_2 + I$ is  linear combination of  elements of the type 
	\[ y_{i_1} \ldots y_{i_r} [y_{j_1},\ldots,y_{j_s}][y_{k_1},y_{k_2}] + I\] where
	\begin{itemize}
		\item[1)] $r\geq 0$, ${k_2} < {k_1}$, ${j_1} > {j_2} \leq \ldots \leq {j_s}$ and  ${i_1} \leq \ldots \leq {i_r}$.		
		\item[2)] ${k_2} \leq {i_1}, {j_2}$.  
		\item[3)] $k_1 \leq j_1$.
		\item[4)] Every $y_{i_1} \ldots y_{i_r} [y_{j_1},\ldots,y_{j_s}][y_{k_1},y_{k_2}]$ has the same multidegree as $v_1 v_2$.
	\end{itemize}
\end{corollary}
\begin{proof} 
	By Proposition \ref{propo dois comutadores} we can suppose without loss of generality that 
	\[ v_1 v_2 = [y_{j_1},y_{j_2}, \ldots,y_{j_n}][y_{k_1},y_{k_2}] \] where ${k_2} < {k_1}$, ${j_1} > {j_2} \leq \ldots \leq {j_n}$ and
	${k_2} \leq  {j_2}$. Suppose that $k_1 > j_1$. By Proposition \ref{propo identidade geral de Y} we have to $v_1 v_2 + I = g - h + I$, where
	\begin{align*}
	g = [y_{k_1}, y_{j_2}, \ldots, y_{j_n}] [y_{j_1},y_{k_2}],\ h = [y_{k_1},y_{j_1},y_{j_3} \ldots, y_{j_n}] [y_{j_2},y_{k_2}].
	\end{align*}
	If $j_1 > j_3$, the Jacobi's identity 
	$ [y_{k_1},y_{j_1},y_{j_3}] = [y_{k_1},y_{j_3} ,y_{j_1}] - [y_{j_1},y_{j_3} ,y_{k_1}]$ it can be applied in $h$. This proof is complete.
\end{proof}

\begin{proposition}\label{propo li com z}
	Let $\Omega_z$  be the subset of  $F\left\langle Y \cup Z \right\rangle $ of multilinear polynomials of the type 
	\[ u_{i_1} \ldots u_{i_r} [u_{j_1},\ldots,u_{j_s}][u_k,z_1]\] 
	where, $r \geq 0$, $u_i\in Y \cup Z$,  $u_{i_1} < \ldots < u_{i_r}$ and $u_{j_1} > u_{j_2} < \ldots < u_{j_s}$. 
	
	Then, $ \Omega_z $ is   linearly  independent module $Id(\mathcal{A},*)$.
	\end{proposition}
	\begin{proof}
	Let $f(z_1,\ldots,z_m,y_1,\ldots,y_n)$ be a  linear combination of  polynomials in $\Omega_z$ such that $f\in Id(\mathcal{A},*)$. 
	%It suffice to show that $f=0$. 
	Without loss of generality we can suppose that $f$  multilinear. Write 
	\[ f = \sum_{i=2}^m h_i [z_i,z_1] + \sum_{j=1}^n g_j [y_j,z_1] \] where
	\begin{itemize}
		\item[i)]	$h_i$ and $g_j$ are multilinear polynomials in variables $\{y_1,\ldots,y_n,z_2,\ldots,z_m\} \setminus z_i  $ and $\{y_1,\ldots,y_n,z_2,\ldots,z_m\} \setminus y_j  $ respectively.
		
		\item [ii)]	$h_i$ and $g_j$ are linear combination of polynomials	 of the type
		\begin{equation}\label{base de Udos} u_{i_1} \ldots u_{i_r} [u_{j_1},\ldots,u_{j_s}] \end{equation} 
		with $u_i\in Y \cup Z$, $u_{i_1} < \ldots < u_{i_r}$ and $u_{j_1} > u_{j_2} < \ldots < u_{j_s}$.
	\end{itemize}

	Fix $2 \leq k \leq n$. 
	Let $A_j,B_i \in UT_2(F)$ where $1 \leq j \leq n$, $1 \leq i \leq m$, $i \neq 1,k$. Let $D = \mathbf{e}_{11} - \mathbf{e}_{22}$	and $I_2$  the identity matrix in  $M_2(F)$. Consider the following matrices in 		$\mathcal{A}$ 
	\begin{align*}
	Z_1 = 	\begin{pmatrix} I_2 & 0 \\ 0 & - I_2 \end{pmatrix},\
	Z_k = \begin{pmatrix} 0 & D \\ 0 & 0 \end{pmatrix},\
	Y_j = \begin{pmatrix} A_j & 0 \\ 0 & A_j^* \end{pmatrix}\ \mbox{and} \
	Z_i = \begin{pmatrix} B_i & 0 \\ 0 & - B_i^* \end{pmatrix}. 
	\end{align*}
	Since $f\in Id(\mathcal{A},*)$ we have to $f (Z_1,\ldots,Z_m,Y_1,\ldots,Y_n,) = 0$. By simple inspection we obtain that $[Y_j,Z_1] = 0$ for $1 \leq j \leq m$, $[Z_i,Z_1] = 0$ for all $i \neq k$ and 
	\[ [Z_k,Z_1] = \begin{pmatrix} 0 & - 2D \\ 0 & 0 \end{pmatrix}. \] Thus, by substituting these matrices in $f$, we have to $h_k [Z_k,Z_1] = 0$ and  therefore 
	$$h_k (B_2,\ldots,\hat{B_k},\ldots,B_m,A_1,\ldots,A_n) = 0.$$
	Thereby, $h_k$ seen as element of $F\left\langle X\right\rangle $ is polynomial identity for $UT_2(F)$, by (\ref{base de Udos}) and Lemma \ref{lema base de Udos} we have to $h_k = 0$. Thus   
	\[ f =  \sum_{j=1}^n g_j [y_j,z_1]. \] 
	
	Analogously, given $1 \leq l \leq m$, can be show that $g_l=0$ by considering the following matrices in $\mathcal{A}$:	
	\begin{equation*}
	Y_l = \begin{pmatrix} 0 & I_2 \\ 0 & 0 \end{pmatrix},\
	Z_1 = \begin{pmatrix} I_2 & 0 \\ 0 & - I_2 \end{pmatrix},\
	Y_j = \begin{pmatrix} A_j & 0 \\ 0 & A_j^* \end{pmatrix} \ \mbox{and} \
	Z_i = \begin{pmatrix} B_i & 0 \\ 0 & - B_i^* \end{pmatrix}
	\end{equation*}	 
	where $A_j,B_i \in UT_2(F)$, $1 \leq j \leq n$, $j \neq l$, $1 \leq i \leq m$,  $i \neq 1$. This proof is complete.
	\end{proof}

\begin{lemma}\label{lema matrizes Y}
	Consider the following matrices in $\mathcal{A}^+$, $Y = \begin{pmatrix} 0 & I_2 \\ 0 & 0 \end{pmatrix}$ and 
	$Y_i = \begin{pmatrix} A_i & 0 \\ 0 & A_i^* \end{pmatrix}$ where $A_i \in UT_2(F)$, $1 \leq i \leq n$. Then,
	\begin{itemize}
		\item[1.] For all $n \geq 3$ we have to	
		\[ [Y,Y_3,\ldots,Y_n] [Y_2,Y_1] = \begin{pmatrix} 0 & -[A_2,A_1,A_3,\ldots,A_n] \\ 0 & 0 \end{pmatrix}. \]
		\item[2.] For all $n \geq 4$ we have to \[ [Y_3,\ldots,Y_n,Y] [Y_2,Y_1] = 0. \]
	\end{itemize}
	\end{lemma}

	\begin{proof} 
		We will to show the part 1 only.
	The proof  can be done by induction,  we will omit the case $n=3$. Suppose that
	$n > 3$ and write $[Y,Y_3,\ldots,Y_n] [Y_2,Y_1] = P_1 - P_2$, where 
	\[ P_1 = [Y,Y_3,\ldots,Y_{n-1}] Y_n [Y_2,Y_1] \ \ \mbox{and} \ \ P_2 = Y_n [Y,Y_3,\ldots,Y_{n-1}] [Y_2,Y_1]. \]
	Let $[Y,Y_3,\ldots,Y_{n-1}] = \begin{pmatrix} 0 & \Theta \\ 0 & 0 \end{pmatrix}$, then
	$$[Y,Y_3,\ldots,Y_{n-1}] [Y_2,Y_1] =\begin{pmatrix} 0 & - \Theta [A_2,A_1] \\ 0 & 0 \end{pmatrix}.$$ 
	By induction $\Theta [A_2,A_1] = [A_2,A_1,A_3,\ldots,A_{n-1}]$. Thus 
	\[ P_2 = \begin{pmatrix}
	0 & - A_n [A_2,A_1,A_3,\ldots,A_{n-1}] \\
	0 & 0
	\end{pmatrix} \]
	and
	\[ P_1 = 
	\begin{pmatrix} 0 & \Theta \\ 0 & 0 \end{pmatrix}
	\begin{pmatrix} A_n & 0 \\ 0 & A_n^* \end{pmatrix}
	\begin{pmatrix} [A_2,A_1] & 0 \\ 0 & - [A_2,A_1] \end{pmatrix}
	= \begin{pmatrix} 0 & \Psi \\ 0 & 0 \end{pmatrix} \] where $\Psi = - \Theta A_n^* [A_2,A_1] = - \Theta [A_2,A_1] A_n = - [A_2,A_1,A_3,\ldots,A_{n-1}] A_n$ as desired. 
	\end{proof}
	
\begin{lemma}\label{lema matrizes Y item 3}
	Let $W = \begin{pmatrix} \mathbf{e}_{22} & 0 \\ 0 & \mathbf{e}_{11} \end{pmatrix}$ and let $Y_1,\ldots Y_n$, $n \geq 4$, be arbitrary elements of $\mathcal{A}$. Then 
	\[ [Y_3,\ldots,Y_n,W] [Y_2,Y_1] = [Y_3,\ldots,Y_n] [Y_2,Y_1]. \]
	\end{lemma}

	\begin{proof}
	There exist  $\alpha_i,\beta_i \in F$ and $\Theta_i \in UT_2(F)$ such that  
	\[ [Y_3,\ldots,Y_n] = \begin{pmatrix} \alpha_1 \mathbf{e}_{12} & \Theta_1 \\ 0 & \beta_1 \mathbf{e}_{12} \end{pmatrix}\ \ \mbox{and} \ \  
	[Y_2,Y_1] = \begin{pmatrix} \alpha_2 \mathbf{e}_{12} & \Theta_2 \\ 0 & \beta_2 \mathbf{e}_{12} \end{pmatrix}. \]
	Then,	 $[Y_3,\ldots,Y_n,W] = \begin{pmatrix} \alpha_1 \mathbf{e}_{12} & \Theta_1 \mathbf{e}_{11} - \mathbf{e}_{22} \Theta_1 \\ 0 & - \beta_1 \mathbf{e}_{12}\end{pmatrix}$ and 
	\[ [Y_3,\ldots,Y_n,W] - [Y_3,\ldots,Y_n] = 
	\begin{pmatrix} 0 & \Theta_1 \mathbf{e}_{11} - \mathbf{e}_{22} \Theta_1 - \Theta_1 \\ 0 & - 2 \beta_1 \mathbf{e}_{12}\end{pmatrix}. \]
	Since that $(\Theta_1 \mathbf{e}_{11} - \mathbf{e}_{22} \Theta_1 - \Theta_1) \mathbf{e}_{12} =- \mathbf{e}_{22} \Theta_1 \mathbf{e}_{12} = 0$ this proof is complete.
	\end{proof}

\begin{proposition}\label{propo li com y}
	Let $\Omega_y \subseteq F \left\langle Y \cup Z \right\rangle $ the subset of multilinear polynomials of the type. 
	\[ [y_{j_1},\ldots, y_{j_s}] [y_k,y_1] \] 
	where	 $j_1 > j_2 < \ldots < j_s$ and $j_1 > k$. Then, $ \Omega_y $ is linearly independent module $Id(\mathcal{A},*)$.
	\end{proposition}
	\begin{proof}
	Let	 $n\geq 4$ and $f(y_1,\ldots,y_n)$ linear combination of polynomials in $\Omega_y$ such that $f\in Id(\mathcal{A},*)$.  Without loss of generality we can suppose that  $f$ multilinear. Write $f = f_4 + \ldots + f_n$ where  
	\[ f_t = \sum_{k=2}^{t-1} \alpha_k^{(t)} [y_t,y_{k_4},\ldots, y_{k_n}] [y_k,y_1] \ \ \ (4 \leq t \leq n) \ (\alpha_k^{(t)} \in F) \] 
	Observe that the indices $(k_4 < \ldots < k_n)$ are uniquely determined by $k$ and $t$.

	Consider the following elements in $\mathcal{A}^+$: 
	$$Y_n = \begin{pmatrix} 0 & I_2 \\ 0 & 0 \end{pmatrix}\  
	\mbox{and}\	Y_j = \begin{pmatrix} A_j & 0 \\ 0 & A_j^* \end{pmatrix}$$ 
	where  $j < n$.
	If $n = 4$ we have to 
	\[f = f_4 = \alpha_2^{(4)} [y_4,y_3] [y_2,y_1] + \alpha_3^{(4)} [y_4,y_2] [y_3,y_1]. \] 
	By Lemma \ref{lema matrizes Y} item 1, we obtain  $\alpha_2^{(4)}  [A_2,A_1,A_3] + \alpha_3^{(4)} [A_3,A_1,A_2] = 0$, therefore, 	 $\alpha_2^{(4)} = \alpha_3^{(4)} = 0$.
	Let $4 \leq t < n$, then by Lemma \ref{lema matrizes Y} item 2 we have to \[ [Y_t,Y_{k_4},\ldots, Y_{k_{n-1}},Y_n] [Y_k,Y_1] = 0\] for all $k<t$. Thus $f_t (Y_1, \ldots, Y_n) = 0$ for all $4 \leq t < n$ and
	\[ f (Y_1, \ldots, Y_n) = f_n (Y_1, \ldots, Y_n) = \sum_{k=2}^{n-1} \alpha_k^{(n)} [Y_n,Y_{k_4},\ldots, Y_{k_n}] [Y_k,Y_1] = 0.\]
	Now, by Lemma \ref{lema matrizes Y} item 1 again, we conclude that 
	\[ \sum_{k=2}^{n-1} \alpha_k^{(n)} [A_k,A_1,A_{k_4},\ldots,A_{k_n}] = 0 \]
for all $A_1,\ldots,A_{n-1} \in UT_2(F)$.	By Lemma \ref{lema base de Udos} we have to $\alpha_k^{(n)} = 0$ for all $2\leq k < n$. 
Therefore,	$f = f_4 + \ldots + f_{n-1}$. 

	Define $g = g_4 + \ldots + g_{n-1}$ where
	\[ g_t (y_1,\ldots,y_{n-1}) = \sum_{k=2}^{t-1} \alpha_k^{(t)} [y_t,y_{k_4},\ldots, y_{k_{n-1}}] [y_k,y_1]. \] where $4\leq t<n$.
	We claim that $g \in Id(\mathcal{A},*)$. In fact, let $ W = \begin{pmatrix} \mathbf{e}_{22} & 0 \\ 0 & \mathbf{e}_{11} \end{pmatrix}$ then by Lemma \ref{lema matrizes Y item 3}, for every $Y_1,\ldots,Y_n \in \mathcal{A}^+$ and every $4\leq t<n$ we have to $$g_t (Y_1, \ldots, Y_{n-1}) = f_t (Y_1, \ldots, Y_{n-1}, W).$$  

	Since $g \in Id(\mathcal{A},*)$ is a polynomial in $n-1$ variables, by induction, we obtain that $\alpha_k^{(t)} = 0$ for all  $2 \leq k < t < n$. This proof is complete.
	\end{proof}

\begin{theorem}\label{thm}
	Let $F$ be a field of characteristic char$F=0$. Then $Id(\mathcal{A},*) = I$. 
	\end{theorem}
	\begin{proof} 
We will to show that $Id(\mathcal{A},*) \subseteq I$.
	Let $f(z_1,\ldots,z_m,y_1,\ldots,y_n) \in Id(\mathcal{A},*)$ be a $Y$-proper multilinear polynomial. Since char$F=0$, its suffices to show that $f \in I$. Since $v_1 v_2 v_3 \in I$ we can write \[ f + I = f_1 + f_2 + I, \] where
	
	\noindent i)  $f_1$ is  multilinear polynomial and  linear combination of 
\begin{align}\label{equacao do teorema}
	z_1^{r_1} \ldots z_m^{r_m} [u_{j_1}, \ldots, u_{j_s}]^{\theta} 
	\end{align} where $0 \leq r_1,\ldots,r_n$ are integers, $\theta \in  \{0,1\}$ and $u_{j_1}, \ldots, u_{j_s} \in Y\cup Z$.
	
	\noindent ii)  $f_2$ is a $Y$-proper multilinear polynomial such that each of its terms has two commutators.
	
	In addition, by Lemma \ref{lema comutador ordenado}, we can suppose that $u_{j_1}> u_{j_2} < \ldots < u_{j_s}$ in (\ref{equacao do teorema}). Since $I \subseteq Id(\mathcal{A}, *)$, see (\ref{primeira inclusao}), we have to $f_1 + f_2 \in Id(\mathcal{A}, *) \subseteq Id(\mathcal{B}, *)$.
	By (\ref{lema B}) it follows that $f_2 \in Id(\mathcal{B}, *)$. Then $f_1 \in Id(\mathcal{B}, *)$. By Proposition \ref{propo base de B} we have to $f_1 = 0$. Thus $f_2 \in Id(\mathcal{A}, *)$. Now, we will to show $f_2 \in I$ by considering two cases.
	
	\
	
	\noindent \textbf{Case}  $m \geq 1$. By Proposition \ref{propo dois comutadores}, there exists a multilinear polynomial $$g (z_1,\ldots,z_m,y_1,\ldots,y_n)$$ such that $f_2 + I = g + I$ and $g$ is a linear combination of polynomials of the type 
	\[ u_{i_1} \ldots u_{i_r} [u_{j_1},\ldots,u_{j_s}][u_k,z_1]\] 
	where, $r \geq 0$, $u_i\in Y \cup Z$,  $u_{i_1} < \ldots < u_{i_r}$ and $u_{j_1} > u_{j_2} < \ldots < u_{j_s}$. Since $f_2 \in Id(\mathcal{A}, *)$, we have to $g \in Id(\mathcal{A}, *)$, by Proposition \ref{propo li com z} it follows that $g = 0$.
	
	\
	
	\noindent \textbf{Case}  $m = 0$. By Proposition \ref{coro so variaveis Y} there exists a multilinear polynomial $g (y_1,\ldots,y_n)$ such that $f_2 + I = g + I$ and $g$ is linear combination of polynomials of the type 
	\[ y_{i_1} \ldots y_{i_r} [y_{j_1},\ldots,y_{j_s}][y_k,y_1] + I\] where $r \geq 0$,  ${i_1} < \ldots < {i_r}$, ${j_1} > {j_2} < \ldots < {j_s}$ and  
	$k < j_1$.
	Write
	\[ g = \sum_i y_{i_1} \ldots y_{i_r} g_i \qquad (i = (i_1, \ldots, i_r)) \] where each $g_i$ is multilinear polynomial in variables 
	$$\{ y_1,\ldots,y_n \} \setminus \{ y_{i_1},\ldots,y_{i_r} \}.$$  
	Suppose that there exists  $i = (i_1, \ldots, i_r)$ such that $g_i \neq 0$. Choose $r$ as being the maximum integer with this property and put $y_{i_1} =  \ldots =  y_{i_r} = 1$ in $g$. Then $g_j = 0$ for all $j \neq i$. Since  $g \in Id(\mathcal{A}, *)$ we have to $g_i \in Id(\mathcal{A}, *)$. By  Proposition \ref{propo li com y} we have a contradiction. 
	
	Therefore, $g = 0$ and so $f_2 \in I$.
\end{proof}

\section{Concluding Remarks}
The Theorem \ref{thm} show that the $T(*)$-ideal of the $*$-polynomial identities of $\mathcal{A}$ is finitely generated because the $T(*)$-ideal generated by any commutator of $F \left\langle Y \cup Z \right\rangle$ is describe completely by 
\[ [z_1,z_2] \quad \mbox{or} \quad [y_1,y_2] \quad \mbox{or} \quad [y_1,z_1]. \]

Over the study of the $*$-polynomial identities of $UT_4(F)$, can be  shown that 
 $$[y_1,z_1] [y_2,z_2] [y_3,z_3]$$ is not the $*$-polynomial identity for $UT_4(F)$. So there exist more  $*$-polynomial identities to be discovered.


\begin{thebibliography}{99}

\bibitem{vinkossca} O. M. Di Vincenzo, P. Koshlukov, R. La Scala.
\textit{Involutions for upper triangular matrix algebras},
Advances in Applied Mathematics 37, (2006) 541-568.





\bibitem{drenskybook}
V. Drensky. \textit{Free algebras and PI-algebras. Graduate Course in
Algebra}, Springer, Singapore, 1999.


\bibitem{drenskygiambruno}
V. Drensky, A. Giambruno. \emph{Cocharacters, codimensions and Hilbert series of the polynomial identities for $2 \times 2$ matrices with involution.}
Canad. J. Math. 46 (1994) 718-733.



\bibitem{urure} D. J. Gonçalves, R. I. Urure, 
\textit{Identities with involution for $2\times 2$ upper triangular matrices algebra over a finite field}

\bibitem{marcoux} L. Marcoux, A. Sourour, \textit{Commutativity preserving linear maps and Lie automorphisms of triangular matrix algebras},
Linear Algebra Appl. 288 (1999) 89–104.
\end{thebibliography}
\end{document}